\documentclass[a4paper,11pt,colorinlistoftodos,final]{amsart}

\makeatletter
\providecommand\@dotsep{4.5}

\makeatother
\usepackage[t,lf]{spectral}
\usepackage[T1]{fontenc}
\usepackage[english]{babel}
\usepackage{amssymb,amscd,amsmath,amsthm,mathtools}
\usepackage[letterpaper, total={8.5in, 11in},
top=0.75in, inner=1.5in, outer=1.5in,bottom=0.75in
]{geometry}
\usepackage{thmtools}
\usepackage[colorlinks=true,linktocpage=true,citecolor=blue,unicode,hyperindex,breaklinks]{hyperref}
\usepackage[capitalize,noabbrev]{cleveref}

\usepackage{csquotes}

\usepackage{xspace}
\usepackage{dutchcal}
\usepackage{mathrsfs}
\usepackage{ upgreek }
\usepackage{eufrak}
\usepackage[shortlabels]{enumitem}
\usepackage{ dsfont }
\usepackage{mathbbol}
\usepackage{stmaryrd}
\usepackage{float}
\usepackage{tikz}
\usetikzlibrary{cd}
\usepackage{moreenum}
\usepackage{animate}
\usepackage{array}
\usepackage{phoenician}
\usepackage{xstring}
\usepackage{xargs}
\usepackage{adjustbox}
\usepackage{appendix}
\usepackage{datetime}
\usepackage{newtxtext, newtxmath}
\usepackage[obeyFinal]{todonotes}
\usepackage{quiver}

\newdateformat{monthyeardate}{\monthname[\THEMONTH] \THEYEAR}

\usetikzlibrary{babel}
\newcommand{\field}[1]{\mathbb{#1}}

\newcommand{\mc}[1]{\ensuremath{\mathcal{#1}}}
\newcommand{\ul}[1]{\underline{#1}}
\newcommand{\wt}[1]{\widetilde{#1}}

\addto\captionsenglish{}
\theoremstyle{definition}
\newtheorem{definition}{\ul{Definition}}[section]
\crefname{definition}{Definition}{Defininitions}
\newtheorem{proposition}[definition]{\ul{Proposition}}
\newtheorem{lemma}[definition]{\ul{Lemma}}

\newtheorem{corollary}[definition]{\ul{Corollary}}
\newtheorem{example}[definition]{\ul{Example}}
\newtheorem{construction}[definition]{\ul{Construction}}
\crefname{construction}{Construction}{Constructions}
\newtheorem{theorem}[definition]{\ul{Theorem}}
\newtheorem{conjecture}[definition]{\ul{Conjecture}}

\newtheorem*{theorem*}{\ul{Theorem}}
\theoremstyle{remark}
\newtheorem{remark}[definition]{\ul{Remark}}
\newtheorem{notation}[definition]{\ul{Notation}}
\newtheorem*{remark*}{\ul{Remark}}
\newtheorem*{notation*}{\ul{Notation}}

\def\Cc{\mc{C}}

\def\C{\field{C}}

\def\sse{\subseteq}

\def\Z{\mathbb{Z}}
\def\R{\field{R}}

\DeclareMathOperator{\Stein}{\mc{Stein}}

\DeclareMathOperator{\tame}{tame}

\DeclareMathOperator{\QCoh}{QCoh}

\DeclareMathOperator{\Spec}{Spec}

\DeclareMathOperator{\Sp}{Sp}

\DeclareMathOperator{\pd}{\partial}

\DeclareMathOperator{\Liq}{Liq}
\DeclareMathOperator{\Nuc}{Nuc}
\DeclareMathOperator{\dStk}{dStk}

\def\Com{\mc{Com}}

\DeclareMathOperator*{\colim}{colim}

\DeclareMathOperator{\Hom}{Hom}

\DeclareMathOperator{\Mod}{Mod}

\DeclareMathOperator{\CAlg}{CAlg}
\def\EFC{{\operatorname{EFC}}}

\def\O{\ensuremath{\mathcal{O}}}

\def\LL{\mathbb{L}}

\def\I{\ensuremath{\mathcal{I}}}
\def\vp{\varphi}
\def\A{\ensuremath{\mc{A}}}

\def\M{\ensuremath{\mathcal{M}}}

\def\Set{\mc{Set}}

\def\Spc{\texttt{Ani}}
\def\Sp{\mc{Sp}}
\newcolumntype{L}{>{$}l<{$}}
\def\temp{&}
\catcode`&=\active
\let&=\temp
\def\AA{\mathbb{A}}

\def\inf1{(\infty,1)}

\def\CI{\mc{C}^\infty}

\def\X{\mc{X}}

\def\web#1{\href{#1}{web}}
\DeclareRobustCommand{\phalphnum}[1]{
	\IfEqCase{#1}{%
		{1}{{\textphnc{a}}}%
		{2}{{\textphnc{b}}}%
		{3}{\textphnc{g}}%
		{4}{\textphnc{d}}%
		{5}{\textphnc{h}}
		{6}{\textphnc{f}}
		{7}{\textphnc{w}}
		{8}{\textphnc{z}}
		{9}{\textphnc{H}}
		{10}{\textphnc{T}}
		{11}{\textphnc{y}}
		{12}{\textphnc{k}}
		{13}{\textphnc{l}}
		{14}{\textphnc{m}}
		{15}{\textphnc{n}}
		{16}{\textphnc{s}}
		{17}{\textphnc{o}}
		{18}{\textphnc{p}}
		{19}{\textphnc{x}}
		{20}{\textphnc{q}}
		{21}{\textphnc{r}}
		{22}{\textphnc{S}}
		{23}{\textphnc{t}}
	}[\PackageError{\phalphnum}{Undefined option to phalphnum: #1}{}]%
}
\def\doi#1{\href{https://doi.org/#1}{doi:#1}}
\def\arXiv#1{\href{https://arxiv.org/abs/#1}{arXiv:#1}}

\title{Liquid Tannaka duality I: classical case}
\author{Waleed Qaisar}
\thanks{}
\author{Gregory Taroyan}
\thanks{The second author's work was supported by Vanier Canada Graduate Scholarship, funding number CGV — 192668}

\date{\monthyeardate\today}

\begin{document}

\maketitle

    \begin{abstract}
		We prove a Tannaka duality statement for geometric stacks in the setting of \textit{analytic} stacks modelled on globally finitely presented Stein spaces. The key ingredient is the theory of liquid vector spaces and liquid quasicoherent sheaves of Clausen--Scholze. As an application, we reconstruct the topological fundamental group of any complex algebraic variety from its category of liquid local systems. We also reconstruct a series of ``twisted fundamental groupoids'' whose representations correspond to meromorphic flat connections on $\mathbb{A}^1$ with logarithmic or irregular singularities at the origin.
	\end{abstract}
	\setcounter{tocdepth}{1}
	\tableofcontents
	\setcounter{tocdepth}{2}
	\listoftodos{}
\section{Introduction}\label{sec:introduction}
In this paper, we prove a Tannaka duality statement for geometric complex analytic stacks; namely, we prove that these can be reconstructed from their categories of \emph{liquid} quasicoherent sheaves.

We work with geometric stacks modelled on globally finitely presented Stein spaces equipped with the \'etale Grothendieck topology. Examples of such stacks are given by classifying stacks of Stein complex Lie groups, including discrete countable groups like the topological fundamental group of any complex algebraic variety. Thus, as an application of the main result, we reconstruct the topological fundamental group of a complex algebraic variety from its category of liquid local systems. We also apply the result to reconstruct a series of ``twisted fundamental groupoids'' whose representations correspond to meromorphic flat connections on $\mathbb{A}^1$ with prescribed singularity behaviour at the origin, including both logarithmic and irregular singularities. 

The main result of the paper is the following:
\begin{theorem*}[\Cref{thm:TD_for_1_geometric_stacks}]
	Let \(\X\) be a geometric analytic stack over \(\C\) modelled on globally finitely presented Stein spaces with the \'etale topology. Then for any derived globally finitely presented Stein space \(S\), the functor
	\[	\text{Hom}(S, \X) \to \text{Hom}^{\tame}_{\Liq_p(\C)}(\QCoh(\X), \QCoh(S))
	\]
	given by sending a morphism \(F:S\to \X\) to the pullback functor \(F^*:\QCoh(\X)\to \QCoh(S)\) is an equivalence.
\end{theorem*}

Here the category \(\text{Hom}^{\tame}_{\Liq_p(\C)}(\QCoh(\X), \QCoh(S))\) is the category of \emph{tame} monoidal functors between the categories of $p$-liquid quasicoherent sheaves on \(\X\) and \(S\) linear with respect to the category of liquid vector spaces. The notion of tameness is a technical condition\footnote{As was shown in \cite{Stefanich_Tannaka}, this condition can be dropped in the algebraic setting; however, we prefer to keep it in this work for the sake of simplicity.} that ensures that the functor preserves colimits, flatness, and \(t\)-structure.

The main new ingredient in our approach is the use of liquid mathematics as developed by Clausen--Scholze in \cite{Clausen_Scholze_Condensed, Clausen_Scholze_Analytic, Clausen_Scholze_complex}. Liquid mathematics provides a natural notion of complex analytic \emph{quasicoherent sheaf}. As we demonstrate, the mere existence of this notion is very powerful. In particular, it allows us to extend Lurie's Tannaka duality argument for geometric stacks in algebraic geometry \cite{Lurie_Tannaka_GS} to complex analytic geometric stacks. This is possible because now one can consider pushforward along general holomorphic maps and not just those with finite fibres as was the case for the classical theory of \emph{coherent} analytic sheaves of Cartan--Oka (see \cite{Grauert_Remmert}). 

We note that the framework we use to study analytic stacks is very close to the one of Porta--Yue Yu \cite{Porta_Yu_GAGA, Porta_Yu_Rep} as is shown in \cite[\S 9.2.1.1]{BBKK_Perspective}, and less general than the approach of \cite{Clausen_Scholze_complex}. The reason is that the latter approach does not treat many objects of interest, such as complex Lie groups, as \emph{affine}. We believe that this difference is superfluous, however, and by importing the more refined set-up of categorified Tannaka duality provided by the work of Stefanich \cite{Stefanich_Tannaka} in the algebraic setting into the analytic setting, one should be able to treat more general analytic stacks in the Clausen--Scholze framework as well. We plan to explore this direction in future work.

Finally, we want to emphasize a point that speaks to the power of the approach we use here: the arguments we use can be adopted without any significant changes in the \(\CI\)-setting as well. Notably in the \(\CI\)-context the affineness restriction is essentially non-existent for finite-dimensional objects and thus our results apply to an enormous class of geometric derived stacks, which in particular includes classifying stacks of \textit{all} Lie groups and Lie groupoids. Building on this observation, we intend in future work to describe the functor of points for the source-simply connected Lie groupoid integrating an integrable Lie algebroid directly in terms of the Lie algebroid, formally via Tannaka duality. This would give a construction of the integration of a Lie algebroid which is very different from the usual approach in differential geometry.
\subsection*{Review of the contents} 
In \Cref{sec:preliminaries} we review the necessary preliminaries on liquid vector spaces and liquid modules. We also review the theory of animated entire functional calculus (EFC) algebras and derived Stein spaces as developed in \cite{Pridham_EFC} and \cite{Pirkovskii}.

In \Cref{sec:efc_liquid_algebras} we introduce the notion of liquid EFC algebras based on the extension property along maps from free EFC algebras. We then relate this category with the category of animated EFC algebras introduced in \Cref{sec:preliminaries}. This is a technically important step since it allows us to relate the theory of EFC algebras defined in extrinsic categorical terms with the intrinsic theory of EFC algebras inside of the category of liquid vector spaces. We then prove an implicit function theorem for liquid EFC algebras. This provides an essential geometric input to our main argument, since it enables us to characterize smooth morphisms of derived Stein spaces in terms of their cotangent complex.

In \Cref{sec:duality_results} we prove the main Tannaka duality results of the paper. We start with \Cref{subsec:tannaka_duality_stein} where we prove a Tannaka duality statement for derived Stein spaces and also show that a version of the ``automatic continuity result'' of Forster \cite{forster1966} holds in this setting. That is we show that for derived Stein spaces one may restrict attention to functors enriched in usual complex vector spaces rather than liquid vector spaces. In \Cref{subsec:tannaka_duality_geometric} we prove the main Tannaka duality statement for geometric analytic stacks. The proof is based on the approach of Lurie \cite{Lurie_Tannaka_GS} adapted to the analytic setting. The key difference is that we have to keep track of the analytic structure using the liquid formalism and the intrinsic characterization of EFC algebras developed in \Cref{sec:efc_liquid_algebras}.

\Cref{sec:applications} contains applications of our Tannaka duality results. In \Cref{subsec:reconstructing_stein_lie_groups} we show how to reconstruct any Stein complex Lie group using our Tannaka duality statement. In \Cref{subsec:fundamental_group} we show how to reconstruct the topological fundamental group of any complex algebraic variety, or more generally of any finite CW complex, using our Tannaka duality result. This is in contrast with the algebraic Tannaka duality formalism which only reconstructs the pro-algebraic completion of the fundamental group. Finally, in \Cref{subsec:stokes_groupoids} we show how to reconstruct a series of stacks classifying meromorphic flat connections on $\mathbb{A}^1$ with prescribed singularity behaviour at the origin, including both logarithmic and irregular singularities.
\subsection*{Acknowledgements}
We would like to thank the following people for helpful comments and their interest in our work: Grigory Andreychev, Joseph Ayoub, Alexander Beilinson, Dustin Clausen, Marco Gualtieri, Owen Gwilliam, Jacob Lurie, Akhil Mathew, Kumar Murty, Madhav Nori, Mauro Porta, Nick Rozenblyum, Peter Scholze and Bertrand To\"en.

The second author would like to thank hirs partner Alisa Chistopolskaya for her constant moral support and encouragement.
\subsection*{Conventions}
We adopt an \(\infty\)-categorical perspective throughout the text. We refer the reader to \cite{Lurie_HTT} and \cite{Lurie_HA} for the foundations of \(\infty\)-categories. We denote by \(\Spc\) the \(\infty\)-category of spaces (i.e. \(\infty\)-groupoids); we also refer to this category as the category of \emph{anima}. We denote by \(\CAlg(\Spc)\) the \(\infty\)-category of commutative algebras in anima (i.e. \(E_\infty\)-algebras). For a commutative algebra \(A\in \CAlg(\Spc)\), we denote by \(\Mod_A\) the stable \(\infty\)-category of \(A\)-modules. We refer the reader to \cite{GaRo1} and \cite{GaRo2} for the foundations of higher algebraic geometry.

We use the language of condensed and liquid mathematics as developed by Clausen--Scholze in \cite{Clausen_Scholze_Condensed, Clausen_Scholze_Analytic,Clausen_Scholze_complex}. We use the terms derived and animated interchangeably. For all of our animated algebras, we use cohomological grading conventions. That is, for a connective animated algebra \(A\), the homotopy groups \(\pi_i(A)\) are concentrated in non-positive degrees.

Finally, we use the language of entire functional calculus (EFC) algebras, see \cite{Pridham_EFC} for a discussion of this formalism. These correspond to Stein spaces in complex analytic geometry. In particular, to simplify the exposition, we mostly work with finitely generated EFC algebras, which correspond to \emph{globally finitely presented} Stein spaces. In fact, most of the results we prove hold for arbitrary theories consisting of idempotent algebras. For instance, a similar method gives Tannaka duality for derived \(\CI\)-stacks. However, we prefer to focus on the ``leading example'' of stacks on globally finitely presented derived Stein spaces in this paper, since it already contains all the main ideas of the general case. In particular, this approach to analytic stacks is closely related to the approach of Porta--Yue Yu \cite{Porta_Yu_GAGA, Porta_Yu_Rep}. 

In future work, we plan to explore the general case in more detail. In particular, we plan to study the connection with other approaches to derived analytic geometry, mainly that of Clausen--Scholze \cite{Clausen_Scholze_complex}.
\section{Preliminaries}\label{sec:preliminaries}

\subsection{Liquid vector spaces and liquid modules}\label{subsec:liquid_vector_spaces}
\begin{definition}
    A condensed anima is a sheaf of homotopy types on the pro-étale site of a point. A condensed set is a \(0\)-truncated (static) condensed anima.
\end{definition}
Below we define the condensed abelian group of ``less-than-\(p\)-measures'' on a given profinite set. This is the basic building block of the theory of \(p\)-liquid vector spaces. 
\begin{definition}
    Fix a real parameter \(0<p\leq 1\) once and for all. 
    \begin{enumerate}
        \item Let \(S\) be a finite set. We define the condensed abelian group \(\M_p(S)\) as the following colimit in condensed sets:
        \[
        \M_p(S) = \colim_{C\in \R_{>0}} \M(S)_{\ell^p\le C}
        \]
        Here \(\M(S)_{\ell^p\le C}\) is the condensed set of signed measures on \(S\) with the \(\ell^p\)-norm bounded by \(C\).
        \item Let \(S\) be a profinite set presented as a limit of finite sets \(S=\lim_i S_i\). We define the condensed abelian group \(\M_p(S)\) as follows:
        \[
            \M_p(S)=\lim_i \M_p(S_i).
        \]
        \item For a profinite set \(S\) define the measure space \(\M_{<p}(S)\) as follows:
        \[
            \M_{<p}(S)=\bigcup_{q<p} \M_q(S).
        \]
    \end{enumerate}
\end{definition}
\begin{definition}\label{def:liquid_groups}
    A condensed abelian group \(V\) is called a \emph{\(p\)-liquid vector space} if it satisfies the following extension property:
    \[
        \begin{tikzcd}
            S \arrow[rr]\arrow[dr] & & V\\
            & \M_{<p}(S) \arrow[ur,dashed,"\text{exists and unique}",swap]
        \end{tikzcd}
    \]
\end{definition}
\begin{remark}
	\Cref{def:liquid_groups} implies that a \(p\)-liquid condensed abelian group is an \(\R\)-condensed vector space. This is shown in \cite[Theorem 3.11]{Clausen_Scholze_complex}.
\end{remark}
\begin{definition}
	We denote by \(\Liq_p(\R)\) the stable category of \(p\)-liquid vector spaces. Algebra objects inside this category are called \emph{liquid $\mathbb{R}$-algebras}. Since $\mathbb{C}$ is a liquid $\mathbb{R}$-algebra, one can take the category of modules for $\mathbb{C}$ inside \(\Liq_p(\R)\) and define this to be the category of \emph{complex \(p\)-liquid vector spaces}, which is denoted \(\Liq_p(\C)\). Finally, one calls algebra objects in \(\Liq_p(\C)\) liquid $\mathbb{C}$-algebras.
\end{definition}
\begin{definition}
	Given a liquid \(\C\)-algebra \(A\) we denote by \(\Mod_A\) its category of modules in \(\Liq_p(\C)\).
\end{definition}

\subsection{Liquid structure on entire holomorphic functions}\label{subsec:liquid_structure_holomorphic}
In this section, we review the construction of entire holomorphic functions on the affine space over \(\C\) as a \(p\)-liquid algebra. The crux of this approach is the same as in \cite[Lecture 4]{Clausen_Scholze_complex}, but some details are slightly different since we are considering functions on the affine space over \(\C\) instead of some finite radius polydisk.
\begin{theorem}[{\cite[Theorem 5.1]{Clausen_Scholze_complex}}]
            There is a (necessarily unique) sheaf $\mathcal{O}$ of $p$-liquid $\mathbb{C}$-algebras on the topological space $\mathbb{C}$ such that for any open disk $D = D(x,r) = \{z \in \mathbb{C}\ | \  |z-x|<r\} \subset \mathbb{C} $,
            one has $\mathcal{O}(D) = \{ \sum_{n=0}^{\infty} a_n(T-x)^n \ | \ a_n \in \mathbb{C}, \forall r'<r, a_n r'^n \rightarrow 0 \}$, with the obvious transition maps for $D' \subset D$.
            
            Moreover for any such disc $D$, the sheaf cohomology groups $H^i(D, \mathcal{O}) = 0$ for $i > 0$.
        \end{theorem}
\begin{construction}\label{con:liquid_struct_on_EF}
        By the above to describe the liquid structure on the algebra of holomorphic functions on $\mathbb{C}$, we just need to produce the global sections of this sheaf of liquid algebras.

        Since we can present $\mathbb{C}$ as a colimit of disks by taking a nested union $\mathbb{C} = \bigcup_{r=1}^{\infty} D(0,r)$, the global sections of this sheaf will be given as the limit $\mathcal{O}(\mathbb{C}) = \varprojlim \mathcal{O}(D(0,r)) = \{ \sum_{n=0}^{\infty} a_nT^n \ | \ a_n \in \mathbb{C}, \forall r, a_n r^n \rightarrow 0 \}$ (noting that liquid vector spaces are closed under limits, and so this limit will also be a liquid algebra).

        Similarly, the liquid structure on the algebra of holomorphic functions on $\mathbb{C}^n$ can be given as the one coming as the limit of liquid algebras of holomorphic functions on polydisks $D^n$ which present $\mathbb{C}^n$ as a nested union. 

        Alternatively, one could use the fact that the algebra of entire holomorphic functions is nuclear Fr\'echet to produce the same liquid structure on it.
\end{construction}
\begin{proposition}
	The algebra of entire holomorphic functions in \(n\) variables \(\O\{t_1,\ldots,t_n\}\) with the above liquid structure is idempotent over \(\C[t_1,\ldots,t_n]\) in the category \(\CAlg(\Liq_p(\C))\), i.e. the canonical map
	\[
		\O\{t_1,\ldots,t_n\}\otimes_{\C[t_1,\ldots,t_n]} \O\{t_1,\ldots,t_n\}\to \O\{t_1,\ldots,t_n\}
	\]
	is an equivalence.
	Moreover, the algebra \(\O\{t_1,\ldots,t_n\}\) is flat over \(\C\) with respect to the \(p\)-liquid tensor product.
\end{proposition}
\begin{proof}
	The proof relies on the fact that \(\O\{t\}\) is a Fr\'echet algebra and thus is an inverse limit along summable maps of measure spaces. Hence, it is flat over \(\C\) with respect to the \(p\)-liquid tensor product, see \cite[Lecture 4]{Clausen_Scholze_complex} for details. Now using the bar complex to calculate the relative tensor product we see that the idempotency follows from the fact that the multiplication map \(\O\{t\}\otimes_{\C[t]} \O\{t\}\to \O\{t\}\) is surjective with kernel generated by the regular sequence \(t\otimes 1 - 1\otimes t.\) The general case follows by induction on \(n.\)
\end{proof}
\subsection{Animated EFC algebras}\label{subsec:animated_efc_algebras}
\begin{definition}
	The (\(1-\)) category \(\O\) is defined as follows.
	\begin{itemize}
		\item The objects of \(\O\) are the finite-dimensional vector spaces over \(\C\).
		\item The morphisms of \(\O\) are the holomorphic maps between them.
	\end{itemize}
	An animated EFC algebra is a functor \(\O\to \Spc\) which preserves finite products. We denote the category of EFC algebras by \(\EFC(\Spc)\).
\end{definition}
\begin{construction}
	There is a functor from the category \(\Com_\C\) of free finitely generated commutative algebras over \(\C\) to the category \(\O.\) This functor induces an adjoint pair between the category of commutative algebras in anima \(\CAlg(\Spc)\) and the category of EFC algebras in anima \(\EFC(\Spc)\). The left adjoint is given by left Kan extension along the above functor. The right adjoint is given by restriction; we also call it the \emph{underlying algebra} functor.
\end{construction}
\begin{definition}
	We denote by \(\O(\C^n)\) the EFC algebra of entire holomorphic functions on \(\C^n\) given by the following functor.
	\[
		\O(\C^n):\O\to \Spc,\quad \C^m\mapsto\C\{T_1,\ldots,T_n\}^{\times m}\in \Set\sse \Spc.
	\]
\end{definition}
\begin{definition}
    A \(\C\)-linear map \(d:A\to M\) from an EFC algebra to a module \(M\) over \(A\) is an \emph{\(\O\)-derivation} if it satisfies the Leibniz rule with respect to all holomorphic operations. That is, we have the following identity for any $k \in \mathbb{N}$, for any entire map $\varphi: \mathbb{C}^k \to \mathbb{C}$:
    \[
        d(\vp(a_1,\ldots,a_k))=\sum_{i=1}^k\frac{\pd \vp}{\pd z_i}(a_1,\ldots,a_k)da_i.
    \]
\end{definition}
\begin{definition}
    The module of EFC K\"ahler differentials is defined as the \(A\)-module corepresenting the functor of \(\O\)-derivations. We denote this module by \(\Omega_{A/\C}.\) More generally, one defines the module of relative EFC K\"ahler differentials \(\Omega_{B/A}\) as the \(B\)-module corepresenting the functor of \(\O\)-derivations which are \(A\)-linear from \(B\) to a \(B\)-module that is an \(A\)-module via the given morphism \(A\to B.\)
\end{definition}
A derived version of the module of EFC K\"ahler differentials can be defined using the cotangent complex formalism in higher algebra \cite[Chapter 7]{Lurie_HA}, as follows:
\begin{definition}
	Let \(A\) be an EFC algebra. We define its category of modules \(\Mod_A\) as the stabilized slice category \(\Sp(\EFC(\Spc)_{/A})\), where $\mathcal{Sp}$ denotes spectra objects. Then the EFC cotangent complex of \(A\) is defined as the object \(\LL_A^\EFC\in \Mod_A\) given by the suspension spectrum of the identity map~\(A\to A\).

	Similarly, for a morphism of EFC algebras \(B\to A\), we define the relative EFC cotangent complex \(\LL_{A/B}^\EFC\in \Mod_A\) as the cofiber of the canonical map \(\LL_B\otimes_B A\to \LL_A\).
\end{definition}
\begin{example}
	The EFC cotangent complex of the EFC algebra \(\O(\C^n)\) is equivalent to the free module \(\O(\C^n)^{\oplus n}\) generated by the differentials \(dT_1,\ldots, dT_n\).
\end{example}
\subsection{Derived Stein spaces and derived holomorphic stacks}\label{subsec:derived_stein_spaces}
\begin{definition}
	We denote by \(\Stein\) the category of globally finitely presented derived Stein spaces, i.e. the dual category of the category of finitely generated EFC algebras.

	This category admits a subcanonical Grothendieck topology generated by families of \'etale maps of Stein spaces which are jointly surjective. The category of derived holomorphic stacks \(\dStk_\O\) is defined as the category of \Spc-valued sheaves on \(\Stein\) with Grothendieck topology as above.
\end{definition}
\begin{remark} As mentioned in the introduction, we restrict the local models to globally finitely presented Stein spaces rather than all Stein spaces to simplify the exposition. The general case should be doable similarly using more general liquid algebras.
\end{remark}
\begin{definition}
	The category \(\QCoh(X)\) of \(p\)-liquid quasicoherent sheaves on a derived holomorphic stack \(X\in \dStk_\O\) is defined as the limit
	\[
		\QCoh(X) = \lim_{S\in \Stein_{/X} } \Mod_{\O(S)}
	\]
	in the \(\infty\)-category of stable \(\infty\)-categories.
\end{definition}

\begin{notation}
    In the above definition, we fix a $p$ once and for all and denote the category simply by QCoh without any mention of $p$, even though we have been referring to $\text{Liq}_p(\mathbb{C)} = \text{QCoh}(*)$ with $p$ explicitly in the notation.
\end{notation}
\section{Liquid EFC algebras}\label{sec:efc_liquid_algebras}
In this section, we introduce and study an auxiliary notion of liquid EFC algebras. As we will see shortly, this notion is equivalent to the notion of an animated EFC algebra.
\subsection{Definition and basic properties}\label{subsec:definition_basic_properties}
Let $\CAlg(\Liq_p(\C))$ denote the category of commutative (\(E_\infty\)-) algebra objects inside the category of liquid vector spaces over $\C$. Below, we will define a full subcategory of $\CAlg(\Liq_p(\C))$, denoted $\EFC\Liq_p(\C)$, of liquid commutative algebras with extra properties of convergence.
\begin{notation}
    We denote by \(\C\{t_1,\ldots,t_n\}\) the algebra of entire holomorphic functions on \(\C^n\) in variables \(t_1,\ldots, t_n\) with the \(p\)-liquid structure coming from the standard structure Fr\'echet structure of convergence on compact subsets.
\end{notation}

\begin{definition}\label{def:absolute_EFC}
    A ``liquid entire functional calculus algebra'' (or liquid EFC algebra) is an object $A \in \CAlg(\Liq_p(\C))$ satisfying the following property for every \(n\in\Z_{\ge 0}:\)
    \[
        \Hom_{\CAlg(\Liq_p(\C)}(\C\{t_1,\ldots,t_n\},A)\xrightarrow{\simeq}\Hom_{\CAlg(\Liq_p(\C)}(\C[t_1,\ldots,t_n],A)
    \]
    Here, the morphism is induced by the inclusion of polynomials into the algebra of entire holomorphic functions. We will refer to the full subcategory of \(\CAlg(\Liq_p(\C)\) spanned by these objects as the category of liquid EFC algebras.
\end{definition}
\begin{lemma}
    The algebra of holomorphic functions on a globally finitely presented Stein space is a liquid EFC algebra.
\end{lemma}

\begin{proof}
    The algebra of holomorphic functions on a globally finitely presented is a finite colimit of algebras of entire holomorphic functions. Since the category of liquid EFC algebras is closed under finite colimits, the result follows.
\end{proof}
Similarly, one defines a relative version of this notion. 
\begin{definition}
	An \emph{EFC point of a liquid commutative algebra \(B\)} is a morphism of liquid commutative algebras \(B\to C\) where \(C\) is a liquid EFC algebra. 

	Given a liquid algebra \(A\), we say that it has \emph{enough EFC points} if there is a collection of EFC points \(\{B\to C_i\}_{i\in I}\) such that the induced base change functors on liquid module categories are jointly conservative.
\end{definition}
\begin{definition}
    Let \(B\to A\) be a morphism of liquid commutative algebras. Then we say that the morphism \(B\to A\) defines a structure of a \emph{relative EFC algebra over \(B\)} if for any EFC point of \(A\to C\) the base change of the morphism \(A\to B\) to \(C\) defines an EFC algebra structure on the liquid algebra \(A\otimes_B C\) over \(C\) in the sense of \Cref{def:absolute_EFC}.
\end{definition}
\begin{definition}
	A \emph{concretely relative EFC algebra} over a liquid commutative algebra \(B\) is a liquid commutative algebra \(A\) over \(B\) such that there exists a presentation of \(A\) as a quotient of a free liquid EFC algebra over \(B\). I.e. there exists a \(\pi_0\)-epimorphism of liquid algebras
	\[
	B\otimes_{\Liq_p}\C\{t_i\}_{i\in S}\to A.
	\]
\end{definition}
\begin{proposition}
	If a liquid commutative algebra \(A\) over a liquid commutative algebra \(B\) is relative EFC and \(B\) has enough EFC points, then \(A\) is concretely relative EFC over \(B\).
\end{proposition}
\begin{proof}
	Let \(\{B\to C_i\}_{i\in I}\) be a collection of EFC points of \(B\) such that the induced base change functors on liquid module categories are jointly conservative. Since \(A\) is relative EFC over \(B\), for each \(i\in I\) there exists a presentation of \(A\otimes_B C_i\) as a quotient of a free liquid EFC algebra over \(C_i\):
	\[
		C_i\otimes_{\Liq_p}\C\{t_j\}_{j\in S_i}\to A\otimes_B C_i.
	\]
	Taking the product over all \(i\in I\) we get a morphism
	\[
		\bigotimes_{i\in I}^B \left(C_i\otimes_{\Liq_p}\C\{t_j\}_{j\in S_i}\right)\to \prod_{i\in I} (A\otimes_B C_i).
	\]
	Since the base change functors are jointly conservative, the above morphism gives a \(\pi_0\)-epimorphism
	\[
		B\otimes_{\Liq_p}\C\{t_j\}_{j\in S} \to A,
	\]
	where \(S=\bigsqcup_{i\in I} S_i.\) This gives the desired presentation of \(A\) as a quotient of a free liquid EFC algebra over \(B.\)
\end{proof}
We make the following somewhat bold conjecture.
\begin{conjecture}
	Every liquid commutative algebra has enough EFC points. Consequently, every relative liquid EFC algebra is a concretely relative EFC algebra.
\end{conjecture}

\subsection{Nuclear spaces and nuclear modules}\label{subsec:nuclear_spaces_modules}
This section is not essential for the proof of the main result; however, we include it here to make a later comment that the main theorem holds if everything is restricted to smaller categories of nuclear objects. In particular, this demonstrates that our results are independent of the choice of parameter \(p\) made in the beginning.
\begin{definition}
	Let \(\Cc\) be a monoidal category. Then we have the following definitions: 
	\begin{enumerate}[(i)]
		\item A morphism \(f:x\to y\) is \emph{trace-class} if it lies in the image of the following canonical map.
	\[
		(x^\vee\otimes y)(*)\to \Hom_\Cc(x,y).
	\]
		\item An object \(c\in \Cc\) is \emph{nuclear} if for all compact objects \(c\in \Cc\) the following natural map is an equivalence.
	\[
		(c^\vee\otimes x)(*)\to \Hom_\Cc(c,x).
	\]
		\item An object \(x \in \Cc\) is \emph{basic nuclear} if it is isomorphic to the colimit of a sequence \(x_0 \to x_1 \to \ldots\) of trace-class maps between objects of \(\Cc\).
	\end{enumerate}
\end{definition}
For applications to quasicoherent sheaves, we will need the following categorical notion.
\begin{definition}[{\cite[I.9.1.2]{GaRo1}}]
	Let $\mathcal{A}$ be a monoidal category. We shall say that $\mathcal{A}$ is \emph{rigid} if the following conditions hold:
    \begin{enumerate}[(i)]
        \item The unit object $\mathbf{1}_\mathcal{A} \in \mathcal{A}$ is compact;
        \item The right adjoint of the multiplication functor $\text{mult}_\mathcal{A}$, denoted $(\text{mult}_\mathcal{A})^R$, is continuous;
        \item The functor $(\text{mult}_\mathcal{A})^R: \mathcal{A} \to \mathcal{A} \otimes \mathcal{A}$ is a functor of $\mathcal{A}$-bimodule categories (a priori it is only a right-lax functor).
    \end{enumerate}
\end{definition}
Below, we summarize several key properties of the category of complex nuclear vector spaces. In particular, we explain why the category of nuclear spaces is rigid. This is essential for the applications to Tannaka duality since it allows us to commute pullback functors with actions by nuclear objects.
\begin{proposition}
	The category \(\Nuc(\C)\) of nuclear objects in \(\Liq_p(\C)\) satisfies the following properties.
	\begin{enumerate}[(i)]
		\item The category \(\Nuc(\C)\) is a full stable subcategory of \(\Liq_p(\C)\).
		\item It is independent of the choice of \(p\). I.e. the categories of nuclear objects in \(\Liq_p(\C)\) are canonically equivalent for all \(p\). 
		\item The category \(\Nuc(\C)\) is rigid.
	\end{enumerate}
\end{proposition}
\begin{proof}
	Below, we indicate the references for each of the properties.
	\begin{enumerate}[(i)]
		\item This is the content of \cite[Theorem 8.6.(1)]{Clausen_Scholze_complex}.
		\item This is the content of \cite[Corollary 8.18]{Clausen_Scholze_complex}.
		\item This follows from \cite[Corollary 4.57]{Ramzi} by considering basic nuclear objects and using the fact that those generate the category of nuclear objects under colimits by \cite[Theorem 8.6.(2)]{Clausen_Scholze_complex}.\qedhere
	\end{enumerate}
\end{proof}
\begin{corollary}\label{cor:lax_module}
	Any lax-module functor of \(\Nuc\)-module categories is strict.
\end{corollary}
\begin{proof}
	This is a basic property of rigid categories. See \cite[Lemma 9.3.6]{GaRo1} for a proof. 
\end{proof}
\begin{definition}\label{def:nuclear_algebra}
    A commutative algebra object in \(\Nuc(\C)\) is called a \emph{nuclear algebra}.
\end{definition}
\begin{lemma}
	The algebra of holomorphic functions on a Stein space is a nuclear liquid algebra.
\end{lemma}
\begin{proof}
	This is immediate from the fact that the algebra of holomorphic functions on the affine space \(\mathbb{A}^n\) is a nuclear liquid algebra and the fact that the category of nuclear liquid algebras is closed under colimits. 
\end{proof}
\subsection{Comparison between animated EFC algebras and liquid EFC algebras}\label{subsec:comparison_animated_liquid}
\begin{notation}
	We use the following notation throughout the rest of this section to distinguish between the two notions of EFC algebras.
	\begin{itemize}
		\item We denote by \(\O\{t_1,\ldots,t_n\}\) the free animated EFC algebra on \(n\) generators.
		\item We denote by \(\C\{t_1,\ldots,t_n\}\) the free liquid EFC algebra on \(n\) generators.
	\end{itemize}
\end{notation}
\begin{theorem}\label{thm:fully_faithfullness_of_EFC}
	There is an equivalence of monoidal \(\infty\)-categories between the category of EFC algebras in anima and the category of liquid EFC algebras.
\end{theorem}
The following argument is essentially due to Ben-Bassat--Kelly--Kremnizer in \cite[Lemma 4.3.36]{BBKK_Perspective}. Another route one could, and probably should, take to deduce this result is to show that the category of liquid vector spaces with the liquid tensor product forms a derived geometry context in the sense of \cite{BBKK_Perspective}, and then apply \cite[Lemma 4.3.36]{BBKK_Perspective} directly. However, we think that the proof is somewhat illuminating, and thus we include it here.
\begin{proof}
	We use the fact that free EFC algebras are idempotent in the category of liquid commutative algebras over the respective free polynomial algebras.

	Let \(A,B\) be two EFC algebras. We may assume that \(A\) is free on \(1\) generator since any \(A\) can be presented as a sifted colimit of such algebras. In addition, \(B\) also admits a presentation as a sifted colimit of free EFC algebras, i.e. \(B\simeq \colim_{i\in \I}\O\{t_1,\ldots,t_{n_i}\}\). Hence, we have the following map.
	\begin{gather*}
		\bigg[\Hom_{\EFC}(\O\{t\},B)\simeq \Hom_{\EFC}(\O\{t\}, \colim_{i\in\I}\O\{t_1,\ldots,t_{n_i}\})\simeq\\\simeq\colim_{i\in\I} \Hom_{\EFC}(\O\{t\}, \O\{t_1,\ldots,t_{n_i}\})\bigg]\to\\ \to \Hom_{\CAlg(\Liq_p(\C))}(\C\{t\},\colim_{i\in\I}\C\{t_1,\ldots,t_{n_i}\})\to\\\to \Hom_{\CAlg(\Liq_p(\C))}(\C[t], \colim_{i\in \I}\C\{t_1,\ldots,t_{n_i}\})\simeq\\\simeq \colim_{i\in\I}\Hom_{\CAlg(\Liq_p(\C))}(\C[t], \C\{t_1,\ldots,t_{n_i}\}).
	\end{gather*}
	By freeness, the composite map is an equivalence. Consequently, the map \[\Hom_{\CAlg(\Liq_p(\C))}(\C\{t\},B)\to \Hom_{\CAlg(\Liq_p(\C))}(\C[t],B)\] is a \(\pi_0\)-epimorphism (since an equivalence has this map as the second factor). However, this map is also a homotopy monomorphism because \(\C\{t\}\) is idempotent over \(\C[t]\). Thus we have established the desired equivalence for free EFC algebras on one generator as the first argument. The general case follows by taking an appropriate sifted colimit.
\end{proof}
\begin{corollary}\label{cor:finitely_generated_embed_fully_faithfully}
	The category of finitely generated EFC algebras embeds fully faithfully into the category of nuclear algebras in \(Liq_p(\C).\) 
\end{corollary}
\begin{proof}
	The proof follows from \Cref{thm:fully_faithfullness_of_EFC} combined with the fact that the category of nuclear algebras (see \Cref{def:nuclear_algebra}) is closed under \(\aleph_1\)-filtered colimits and geometric realizations.
\end{proof}
\begin{remark}
	The category of animated EFC algebras is defined independently of the category of \(p\)-liquid vector spaces. Thus its essential image in \(\CAlg(\Liq_p(\C))\) is also \(p\)-independent. In the case of finitely generated EFC algebras, this also follows from \Cref{cor:finitely_generated_embed_fully_faithfully}.
\end{remark}
\subsection{EFC cotangent complex and liquid cotangent complex}\label{subsec:efc_cotangent_complex}
In this section, we review the equivalence between liquid and EFC infinitesimal theories.
\begin{definition}
	Let \(A\) be a liquid algebra. We define its cotangent complex \(\LL_A\) as the object corepresenting the functor of derivations from \(A\) to liquid \(A\)-modules. We denote this liquid \(A\)-module by \(\LL_A^{\Liq}.\)

	Similarly, given a morphism of liquid algebras \(A\to B\) we define the relative cotangent complex \(\LL_{B/A}^{\Liq}\) as the object corepresenting the functor of \(A\)-linear derivations from \(B\) to a liquid \(B\)-module that is an \(A\)-module via the given morphism \(A\to B.\)
\end{definition}
\begin{proposition}
	Let \(A\) be a liquid finitely presented EFC algebra. Then the underlying liquid \(A\)-module of the EFC cotangent complex \(\LL_A^{\EFC}\) is equivalent to the liquid cotangent complex \(\LL_A^{\Liq}\) of \(A\) computed in \(\Liq_p(A).\)
\end{proposition}
\begin{proof}
	It is clear that \(\LL_A^{\EFC}\) is equivalent to the cotangent complex of \(A\) as a liquid algebra when \(A\) is free and finitely generated as an EFC algebra. The general case follows by taking the appropriate finite colimit.
\end{proof}
\begin{notation}
	From now on, we cease to distinguish between animated EFC algebras and liquid EFC algebras since they are equivalent by \Cref{thm:fully_faithfullness_of_EFC}.
\end{notation}
\subsection{Infinitesimal criterion for smoothness}\label{subsec:infinitesimal_criterion} In this section, we prove an infinitesimal criterion for smoothness for \emph{EFC morphisms} of liquid algebras. The idea is that even though EFC morphisms are not of finite type in the usual algebraic sense, they behave like morphisms of finite type \emph{relative} to the theory of holomorphic functions. Thus, one can hope to adapt the classical infinitesimal criteria for smoothness to this setting. Another viewpoint on this is that the inverse function theorem holds for holomorphic functions, and thus one can hope to bootstrap this to the level of EFC algebras by appropriately interpreting the Jacobian condition. 

This section provides a translation of Steffens' results in the \(\CI\)-case \cite[\S 5.1.3]{Steffens_Thesis} to the EFC setting. This easy translation suggests that the results of Steffens actually hold in a more general setting of Fermat theories, see e.g. \cite{Carchedi_Roytenberg}, which we do not pursue here.

\begin{theorem}[Implicit Function Theorem]\label{st:infinitesimal_criterion}
    Let \(X\to Y\) be a morphism of Stein spaces. Then if the module of relative EFC K\"ahler differentials \(\Omega_{X/Y}\) is locally a (liquid) free module of finite rank, the morphism \(X\to Y\) is a smooth submersion. That is, it is locally of the form \(Y\times \C^k\to Y\). 
\end{theorem}
\Cref{st:infinitesimal_criterion}  follows from the following EFC version of the inverse function theorem.
\begin{theorem}[Inverse Function Theorem]\label{thm:inverse_function_theorem}
	Let \(A\to B\) be a morphism of EFC algebras of finite type. Then it is \'etale if and only if it is formally \'etale, i.e. the relative cotangent complex \(\LL_{B/A}\) vanishes.
\end{theorem}
The following proof is essentially the same as Steffens' in \cite[Proposition 5.1.3.32]{Steffens_Thesis}, adapted to our setting.
\begin{proof}[\ul{Proof of \Cref{st:infinitesimal_criterion} from \Cref{thm:inverse_function_theorem}}]
	One direction is obvious, i.e. if the morphism is locally of the form \(A\to A\otimes_{\Liq_p}\C\{t_1,\ldots,t_n\}\) then the relative cotangent complex is locally free of rank \(n\). For the converse direction, we use the transitivity triangle to get the following fibre sequence.
	\[
		\LL_{A}\otimes_A B\to \LL_{B}\to \LL_{B/A}.
	\]
	By assumption \(\LL_{B/A}\) is a locally free liquid \(B\)-module of finite rank. Thus, we can replace the module \(\LL_{B/A}\) by a direct sum of \(n\) copies of \(B\) in degree \(0.\) We have a map of the following form.
	\[
		\pi_0(B)^m\xrightarrow{K} \pi_0(B)^n\to 0.
	\]
	The minors of the matrix \(K\) generate the unit ideal in \(\pi_0(B)\) by the fitting ideal lemma and thus after a suitable localization can be chosen to determine a formally \'etale map \(A\otimes \C\{T_1,\ldots, T_n\} \to B.\) Now we use \Cref{thm:inverse_function_theorem} to conclude that this map is indeed \'etale. 
\end{proof}
Proof of \Cref{thm:inverse_function_theorem} relies on the following two lemmas.
\begin{lemma}\label{lem:factorization}
	Let \(f: A\to B\) be a morphism of EFC algebras such that the relative cotangent complex is connected, or equivalently, the module of classical Kähler differentials is zero. Then there exists a factorization of the following form for a finite collection of elements \(b_1,\ldots,b_n\in \pi_0(B).\)
	\[
		A\xrightarrow{\pi^*} A_i\xrightarrow{\iota^*} B\{b_i^{-1}\}
	\] 
	Here \(\pi^*\) is an \'etale morphism and \(\iota^*\) is a surjection on \(\pi_0.\) Here \(B\{b_i^{-1}\}\) is the localization of \(B\) at the element \(b_i\) \emph{in the EFC sense}, i.e. the the algebra of holomorphic functions on the open subset of \(\Spec B\) where \(b_i\) does not vanish.
\end{lemma} 
\begin{proof}
    The proof is essentially the same as that of \cite[Lemma 5.1.3.20]{Steffens_Thesis}. We include here a sketch of the argument for the sake of completeness.

	We start by presenting \(A\) and \(B\) as quotients of free EFC algebras as follows.
	\[
		\C\{t_1,\ldots,t_n\}/I\simeq \pi_0(A), \quad \C\{t_1,\ldots,t_n,t_{n+1},\ldots,t_m\}/J\simeq \pi_0(B).
	\]
	By assumption on the K\"ahler differentials we see that the Jacobian matrix has to have rank equal to \(m-n.\) Then we can choose a collection of \(m-n\) generators of the ideal \(J\) such that the corresponding minor of the Jacobian matrix is invertible in \(\pi_0(B).\) Localizing at the determinant of this minor we get the desired factorization at the level of \(\pi_0.\) By doing an appropriate derived factorization (that uses connectivity of the cotangent complex) we can lift this to the level of animated EFC algebras.
\end{proof}
\begin{lemma}[{\cite[Lemma 5.1.3.22]{Steffens_Thesis}}]\label{lem:etale_localization}
	Let \(A\to B\) be a \(\pi_0\)-surjective morphism of EFC algebras such that the relative cotangent complex \(\LL_{B/A}\) is \(1\)-connected (i.e. \(\pi_{0}(\LL_{B/A})=\pi_{1}(\LL_{B/A})=0\)). Then the morphism \(\pi_0(A)\to \pi_0(B)\) is a localization of EFC algebras for some element \(\pi_0(A).\)
\end{lemma}
\begin{proof}
	The proof is standard and relates the connectivity of the cotangent complex with the vanishing of the naive Kähler differentials. From there, one can explicitly construct the desired localizing element. 
\end{proof}
\begin{proof}[\ul{Proof of \Cref{thm:inverse_function_theorem}}]
	The proof of this result is identical to that of \cite[Theorem 5.1.3.17]{Steffens_Thesis}. One direction is obvious, i.e. if the morphism is \'etale then the relative cotangent complex vanishes. For the converse direction, we use \Cref{lem:factorization} to get a factorization of the following form.
	\[
		A\xrightarrow{\pi^*} A_i\xrightarrow{\iota^*} B\{b_i^{-1}\}
	\] 
	Here \(\pi^*\) is an \'etale morphism and \(\iota^*\) is a surjection on \(\pi_0.\) Since the relative cotangent complex of \(A\to B\) vanishes, the same holds for the morphism \(A_i\to B\{b_i^{-1}\}.\) Thus it is enough to show that \(\iota^*\) is \'etale. However now we can use \Cref{lem:etale_localization} to conclude since \(\iota^*\) is a \(\pi_0\)-epimorphism.\qedhere
\end{proof}

\section{Duality results}\label{sec:duality_results}
\subsection{Tannaka duality for Stein spaces}\label{subsec:tannaka_duality_stein} In this section, we show a basic Tannaka duality statement for Stein spaces. We first review the following statement known as Stein duality, see \cite{Pridham_EFC}, \cite{Pirkovskii} for details, which we also provide a proof of for the sake of completeness.
\begin{proposition}\label{st:Stein_Duality}
    Given a morphism of EFC algebras of finite type \(A\to B\) there is an associated morphism of Stein spaces \(\Spec B\to \Spec A.\) 
\end{proposition}

\begin{proof}
	Consider two finite presentations for \(A\) and \(B\). Denote them by \(\C\{T_1,\ldots,T_n\}\xrightarrow{\pi} A\) and \(\C\{T_1,\ldots,T_m\}\xrightarrow{\pi'} B.\) Then we have the following diagram.
	% https://q.uiver.app/#q=WzAsNCxbMCwwLCJcXE8oXFxBQV5tKSJdLFswLDIsIkEiXSxbMiwyLCJCIl0sWzIsMCwiXFxPKFxcQUFebSkiXSxbMCwxLCJcXHBpIiwyXSxbMSwyLCJmIl0sWzMsMiwiXFxwaSciXSxbMCwzLCJcXHd0e2Z9IiwwLHsic3R5bGUiOnsiYm9keSI6eyJuYW1lIjoiZGFzaGVkIn19fV0sWzAsMl1d
\[\begin{tikzcd}
	{\C\{T_1,\ldots,T_n\}} && {\C\{T_1,\ldots,T_m\}} \\
	\\
	A && B
	\arrow["{\wt{f}}", dashed, from=1-1, to=1-3]
	\arrow["\pi"', from=1-1, to=3-1]
	\arrow[from=1-1, to=3-3]
	\arrow["{\pi'}", from=1-3, to=3-3]
	\arrow["f", from=3-1, to=3-3]
\end{tikzcd}\]
	Now we observe that \(\wt{f}\) can be constructed as follows. Consider the generators \(T_1,\ldots,T_n\) of the liquid EFC algebras \(\C\{T_1,\ldots,T_n\}\). Then we have a collection of \(n\) elements \(f\circ\pi(T_1),\ldots,f\circ\pi(T_n)\) in \(B.\) They determine a morphism of liquid algebras of the following form. 
	\[
		\begin{tikzcd}
			& \C\{T_1,\ldots,T_n\} \arrow[dr, dashed] &\\
			\C[T_1,\ldots,T_n] \arrow[rr] \arrow[d, "\pi\circ \iota"] \arrow[ru,"\iota"] & & \C\{T_1,\ldots,T_m\} \arrow[d, "\pi'"] \\
			A \arrow[rr, "f"] & & B
		\end{tikzcd}
	\]
	The dashed arrow exists because \(\C\{T_1,\ldots,T_m\}\) is an EFC algebra and is precisely the map \(\wt{f}\) that we wanted. Thus, we have the desired map \(\Spec B\to \Spec A\) obtained by restriction and corestriction from the map of affine spaces \(\AA^m\to \AA^n\) induced by \(\wt{f}.\)
\end{proof}
The following result is a basic Tannaka duality statement for Stein spaces. It is essentially a direct consequence of Stein duality of \Cref{st:Stein_Duality}.
\begin{theorem}
    Consider a (derived) Stein space \(U.\) Consider a \(\Liq_p(\mathbb{C})\)-enriched symmetric monoidal colimit preserving functor \(F^*\) from \(\QCoh(U)\) to the category \(\QCoh(S)\) for a (derived) Stein space \(S.\) Then there exists a holomorphic map \(F:S\to U\) inducing this functor.
\end{theorem}
\begin{proof}
    Note that since the functor \(F^*\) is monoidal, we have the following morphism on \(\Hom\) liquid vector spaces.
    \[
        \Hom(\O_U,\O_U)\xrightarrow{F^*} \Hom(F^*\O_U,F^*\O_U)=\Hom(\O_S,\O_S).
    \]
    This, in turn, implies that there exists a morphism of liquid commutative algebras \(F^*:\O(U)\to \O(S).\) Note that both \(\O(U)\) and \(\O(S)\) are liquid EFC algebras by virtue of being comprised of holomorphic functions on Stein spaces. Hence, the morphism \(F^*\) preserves the EFC structure because it is a map of liquid algebras. Consequently, we can use Stein duality of \Cref{st:Stein_Duality} to see that there is the requisite map \(S\to U.\)
\end{proof}
The following result which replaces the requirement above that the functor is enriched in liquid vector spaces to just being enriched in usual vector spaces is a consequence of the ``automatic continuity phenomenon'' of Forster \cite{forster1966}. Recently, Benoist obtained an improved version of this result in \cite{Benoist_Stein}, in which one does not have to assume that the Stein spaces in question are of finite dimension. 
\begin{theorem}
	Consider a (derived) Stein space \(U\). Consider a \(\C\)-linear symmetric monoidal colimit preserving functor \(F^*\) from \(\QCoh(X)\) to the category \(\QCoh(S)\) for a (derived) Stein space \(S.\) Then there exists a holomorphic map \(F:S\to U\) inducing this functor.
\end{theorem}
On its face, the result above seems surprising; indeed, one could imagine pathological algebra homomorphisms between algebras of holomorphic functions that are not continuous. However, automatic continuity of functionals with finite-dimensional range implies that such pathological homomorphisms cannot exist for algebras of holomorphic functions on Stein spaces.
\begin{proof}
	By \cite[Theorem 3.2]{Benoist_Stein} any algebra homomorphism between the algebras of holomorphic functions on Stein spaces is automatically continuous. Hence, it is clear that at the level of \(\pi_0\) of algebras of functions, there is a continuous (and hence liquid) homomorphism inducing the functor \(F^*.\) It remains to note that the continuity condition on negative homotopy groups for EFC algebras is vacuous and we may conclude as in the previous theorem.
\end{proof}

\subsection{Tannaka duality for (1-)geometric stacks}\label{subsec:tannaka_duality_geometric} In this section, we prove the main Tannaka duality statement for any (1-) geometric stack. We follow the approach of Lurie \cite{Lurie_Tannaka_GS} in the algebraic setting. In particular, we impose a tameness condition on our functors, rather than model our approach on the more general recent theorem of Stefanich \cite{Stefanich_Tannaka} which drops the tameness hypothesis, since it is more straightforward and requires less category theory to understand. We intend to import Stefanich's method into the analytic setting in future work to get a fully general Tannaka duality result.

\begin{definition}[{\cite{Lurie_Tannaka_GS}, \cite{Lurie_DAG8}}]
	A right \(t\)-exact symmetric monoidal functor between two stable presentable monoidal categories with \(t\)-structures is said to be \emph{tame} if it preserves flatness and small colimits.
\end{definition}

\begin{definition}
	A stack \(\X\) on the category \(\Stein\) of (derived) Stein spaces is \emph{geometric} if the following conditions hold.
	\begin{enumerate}[(i)]
		\item The diagonal morphism \(\Delta:\X\to \X\times \X\) is representable by derived Stein spaces and is a smooth morphism.
		\item There exists a smooth surjective morphism \(p:U\to \X\) from a derived Stein space~\(U.\)
	\end{enumerate}
\end{definition}
\begin{theorem}\label{thm:TD_for_1_geometric_stacks} Let \(\X\) be a geometric stack. Then the following categories are equivalent.
    \[
		T:\Hom(S,\X) \to \Hom^{\tame}_{\Liq_p(\C)}(\QCoh(\X),\QCoh(S)).
    \]
    The right-hand side is the full subcategory of the category of functors spanned by tame functors.

    We note that the same result holds for functors between categories of nuclear objects (see \Cref{subsec:nuclear_spaces_modules}), and that all our constructions remain in the category of nuclear quasicoherent sheaves.
\end{theorem}
Before giving the proof, we want to summarize the essence of our approach. In \cite{Lurie_Tannaka_GS}, Lurie uses the following strategy. Consider the sheaf of algebras corresponding to the pushforward of the structure sheaf on the smooth atlas to the geometric stack \(\X.\) A tame functor out of the category of quasicoherent sheaves on \(\X\) maps this sheaf into a sheaf of algebras over \(S.\) Using the tameness assumption one shows that this sheaf of algebras admits local sections and thus comes from a geometric morphism \(S\to \X.\)

Compared to Lurie's original argument, we have two things to establish. First, we need to show that the sheaf over the space \(S\) we are trying to map out of is ``sufficiently analytic''. Then we can apply the implicit function theorem \Cref{st:infinitesimal_criterion} to establish the existence of a section. Here, we essentially use the existence of the category of quasicoherent sheaves. We start with a ``purely algebraic'' sheaf of algebras over an analytic space, verify that it is analytic, and then show that it has good smoothness properties. All this is enabled by our ability to simultaneously perform algebraic and analytic constructions in the same category of sheaves.
\begin{proof}[Proof of essential surjectivity]
	Denote by \(A\) the liquid EFC algebra corresponding to the smooth atlas \(p:U\to \X\). Denote by \(B\) the liquid EFC algebra corresponding to the diagonal of \(\X\). That is, we have the following diagram of analytic stacks. 
	% https://q.uiver.app/#q=WzAsNCxbMCwwLCJcXFNwZWMgQiJdLFsyLDAsIlxcU3BlYyBBPVUiXSxbMiwyLCJcXFgiXSxbMCwyLCJcXFNwZWMgQSJdLFswLDFdLFsxLDIsInAiXSxbMCwzXSxbMywyXSxbMCwyLCIiLDEseyJzdHlsZSI6eyJuYW1lIjoiY29ybmVyIn19XV0=
\[\begin{tikzcd}
	{\Spec B} && {\Spec A=U} \\
	\\
	{\Spec A} && \X
	\arrow[from=1-1, to=1-3]
	\arrow[from=1-1, to=3-1]
	\arrow["\lrcorner"{anchor=center, pos=0.125}, draw=none, from=1-1, to=3-3]
	\arrow["p", from=1-3, to=3-3]
	\arrow[from=3-1, to=3-3]
\end{tikzcd}\]
Denote by \(\A\) the object \(F^*p_*\O_U\). We observe that the following equivalence of sheaves holds on \(\X\).
\[
p_*\O_U\otimes_{\O_\X} p_*\O_U\simeq p_*\O_U\otimes_A B.
\]
Thus, by symmetric monoidality and cocontinuity of the functor \(F^*\), we have the following equivalence of sheaves on \(S\).
\[
\A\otimes_{\O_S} \A \simeq \A\otimes_{A} B.
\]

Since $\mathcal{X}$ was a geometric stack (i.e. the map $\text{Spec} A \to \mathcal{X}$ is smooth, and thus the map $A \to B$ is smooth), we have that $B$ has a presentation as $A\{x_1, \dots x_n \}\twoheadrightarrow B$ with the (liquid) cotangent complex being a retract of a free finitely generated module. Then we have a presentation:

\[ \mathcal{A}\otimes_{\mathcal{O}_S} {\mathcal{O}_S}\twoheadrightarrow \mathcal{A} \otimes_A B.
\]
Thus we have produced a \(\pi_0\)-epimorphism:
\[\mathcal{A}\otimes_{\mathcal{O}_S} {\mathcal{O}_S}\{x_1, \dots, x_n \} = \mathcal{A}\{x_1, \dots, x_n \} \to \mathcal{A} \otimes_{\mathcal{O}_S} \mathcal{A}.
\]

$\mathcal{A}$ is faithfully flat over $\mathcal{O}_S$ because $p_*(\mathcal{O}_U)$ is faithfully flat since $p$ is a surjective smooth map from an affine, and $F^*$ sends faithfully flat objects to faithfully flat objects. Thus, we have a map which is still a \(\pi_0\)-epimorphism:

\[ \O_S\{x_1, \dots, x_n \} \overset{\eta}{\twoheadrightarrow} \A.
\]

We claim that the existence of this \(\pi_0\)-epimorphism endows $\A$ with an EFC structure. To see this, we need to verify the extension property. 

% https://q.uiver.app/#q=WzAsMyxbMCwwLCJcXG1hdGhjYWx7T31fU1t4XzEsIFxcZG90cywgeF9tIF0iXSxbMiwwLCJcXG1hdGhjYWx7QX0iXSxbMCwxLCJcXG1hdGhjYWx7T31fU1xce3hfMSwgXFxkb3RzIHhfbVxcfSJdLFswLDEsIlxcdmFycGhpIl0sWzAsMl0sWzIsMSwiXFx0aWxkZXtcXHZhcnBoaX0iLDIseyJzdHlsZSI6eyJib2R5Ijp7Im5hbWUiOiJkYXNoZWQifX19XV0=
\[\begin{tikzcd}
	{\mathcal{O}_S[x_1, \dots, x_m ]} && {\mathcal{A}} \\
	{\mathcal{O}_S\{x_1, \dots x_m\}}
	\arrow["\varphi", from=1-1, to=1-3]
	\arrow[from=1-1, to=2-1]
	\arrow["{\tilde{\varphi}}"', dashed, from=2-1, to=1-3]
\end{tikzcd}\]

where we are given $\varphi$ and are looking for the extension $\wt{\varphi}$. Using the \(\pi_0\)-epimorphism $\eta$, we can find such an extension by taking a lift to the EFC algebra covering \(\A\) and then composing with the \(\pi_0\)-epimorphism.

% https://q.uiver.app/#q=WzAsNCxbMCwwLCJcXE9fU1t4XzEsXFxsZG90cyx4X25dIl0sWzAsMiwiXFxPX1NcXHt4XzEsXFxsZG90cyx4X25cXH0iXSxbMiwyLCJcXEEiXSxbMiwwLCJcXE9fU1xce3hfMSxcXGxkb3RzLHhfbVxcfSJdLFswLDFdLFswLDMsIlxcd3R7XFx2cH0iLDAseyJjdXJ2ZSI6LTJ9XSxbMywyLCJcXGV0YSIsMCx7InN0eWxlIjp7ImhlYWQiOnsibmFtZSI6ImVwaSJ9fX1dLFswLDIsIlxcdnAiLDIseyJjdXJ2ZSI6M31dLFsxLDMsIlxcdGV4dHtcXHNjcmlwdHNpemUgZXhpc3RzIGFuZCBpcyB1bmlxdWV9IiwxLHsibGFiZWxfcG9zaXRpb24iOjYwfV0sWzEsMiwiIiwwLHsic3R5bGUiOnsiYm9keSI6eyJuYW1lIjoiZGFzaGVkIn19fV1d
\[\begin{tikzcd}[ampersand replacement=\&]
	{\O_S[x_1,\ldots,x_n]} \&\& {\O_S\{x_1,\ldots,x_m\}} \\
	\\
	{\O_S\{x_1,\ldots,x_n\}} \&\& \A
	\arrow["{\wt{\vp}}", curve={height=-12pt}, from=1-1, to=1-3]
	\arrow[from=1-1, to=3-1]
	\arrow["\vp"', curve={height=18pt}, from=1-1, to=3-3]
	\arrow["\eta", two heads, from=1-3, to=3-3]
	\arrow["{\text{\scriptsize exists and is unique}}"{description, pos=0.6}, from=3-1, to=1-3]
	\arrow[dashed, from=3-1, to=3-3]
\end{tikzcd}\]

Since $A \to B$ is smooth, by base change along the diagram, we see that the cotangent complex of the map $\A \to \A \otimes_A B$ is a retract of a finitely generated free module. i.e. the map is formally smooth.

% https://q.uiver.app/#q=WzAsNCxbMCwwLCJBIl0sWzAsMSwiQiJdLFsxLDAsIlxcbWF0aGNhbHtBfSJdLFsxLDEsIlxcbWF0aGNhbHtBfSBcXG90aW1lc19BIEIiXSxbMCwxXSxbMCwyXSxbMiwzXSxbMSwzXV0=
\[\begin{tikzcd}
	A & {\mathcal{A}} \\
	B & {\mathcal{A} \otimes_A B}
	\arrow[from=1-1, to=1-2]
	\arrow[from=1-1, to=2-1]
	\arrow[from=1-2, to=2-2]
	\arrow[from=2-1, to=2-2]
\end{tikzcd}\]

Again using the isomorphism $\A \otimes_A B \cong \A \otimes_{\O_S} \A$ we have that the map $\A \to \A \otimes_{\O_S} \A$ is formally smooth. Using the faithful-flatness of $\A$ as an $\O_S$-module, we have that the map $\O_S \to \A$ is formally smooth. Now that we have a formally smooth map from $\O_S$ to an EFC algebra $\A$, we use the implicit function theorem (\Cref{st:infinitesimal_criterion}) to finally produce a section $\mathcal{A} \to \O_S$.

Now it is easy to verify (following Lurie's original argument) that the section $\mathcal{A} \to \O_S$ produces the desired map $S \to \mathcal{X}$ inducing the functor $F^*$.
\end{proof}
\begin{proof}[Proof of fully faithfulness]
	The proof is identical to the algebraic case.
\end{proof}
\begin{remark} Here we want to reiterate the point made at the beginning of this section about the generality of our results. Recent work of Stefanich \cite{Stefanich_Tannaka} and Stefanich--Scholze \cite{Stefanich_Scholze_Gestalten} suggests that there should be a fully general Tannaka duality result for higher geometric stacks over Stein spaces. The idea in the algebraic setting is to ``vertically categorify'' the Tannaka duality set-up by replacing rings with appropriate module categories, module \(2\)-categories over those and so on. We believe that this approach should work in the analytic setting as well, and we intend to pursue this in future work.
\end{remark}
\section{Applications}\label{sec:applications}

In this section, we first show that we can reconstruct any Stein complex Lie group G using the Tannakian formalism developed above. As a result of this, we can reconstruct the topological fundamental group of any complex algebraic variety, or more generally of any finite CW complex, rather than just the pro-algebraic completion of this fundamental group as is outputted by the algebraic Tannakian formalism. 

We follow this up by reconstructing from their liquid representations the so-called ``Stokes groupoids'' \cite{Gualtieri_Li_Pym_Stokes}. These are a series of analytic groupoids such that the finite-dimensional (i.e. on coherent sheaves) representations of the $n^{th}$ groupoid are exactly flat connections on $\mathbb{A}^1$ with a pole of order $n$ at the origin (i.e. the local theory of flat connections with a logarithmic singularity for the $n = 1$ case, and irregular singularities for the $n > 1$ case).

We then give an example which shows why there is no analogous analytic Tannaka duality theorem if one restricts attention to just coherent sheaves, followed by an example of a situation which shows that the condition of affineness of the diagonal is necessary for the main theorem.

\subsection{Reconstructing a Stein complex Lie group}\label{subsec:reconstructing_stein_lie_groups}

For a complex Lie group $G$, we denote the category $\text{QCoh}(BG)$ by $\text{Rep}^{\text{Liq}}(G)$. We also denote the pullback functor induced by the map $* \to BG$ by $U^{\text{Liq}}$ and call it the forgetful functor.

\begin{theorem}\label{thm:Stein_analytic_groups}
	Let \(G\) be a Stein complex Lie group. Let $U^{\text{Liq}}: \text{Rep}^{\text{Liq}}(G) \to \text{Liq}_p(\mathbb{C})$ be the forgetful functor to liquid vector spaces. Then $G = \text{Aut}^\otimes(U^{\text{Liq}})$.
\end{theorem}
\begin{proof}
	Since $G$ gives the diagonal in the following diagram

\[\begin{tikzcd}
	G & {*} \\
	{*} & BG
	\arrow[from=1-1, to=1-2]
	\arrow[from=1-1, to=2-1]
	\arrow[from=1-2, to=2-2]
	\arrow[from=2-1, to=2-2]
\end{tikzcd}\]
    and is Stein (hence globally finitely presented by \cite[Th\'eor\`eme 2]{Matsushima}), $BG$ is a geometric stack and \Cref{thm:TD_for_1_geometric_stacks} applies. Taking automorphisms
    of $U^{\text{Liq}}$ then gives:
    \[\text{Aut}^{\otimes}(U^{\text{Liq}}: \text{Rep}^{\text{Liq}}(G) \to \text{Liq}_p(\mathbb{C})) = \text{Aut}(* \to BG) = G
    \]
    where the last equality follows from the fact that such automorphisms are given by the diagonal.
\end{proof}

\begin{example}\label{countable_groups_example}

Note that all analytic matrix Lie groups are of this form. Examples of such Stein complex Lie groups, which are \textit{not} algebraic groups, are discrete countable groups. To see this, note first that such groups have a unique holomorphic structure as a 0-dimensional complex manifold, and they are all in fact biholomorphic to $\mathbb{Z}$. Then note that $\mathbb{Z}$ is a globally finitely presented Stein space as it is the vanishing locus of the holomorphic function $f(z)=\text{sin}(\pi z)$. 
    
\end{example}

\subsection{Reconstructing the topological fundamental group of a complex algebraic variety}\label{subsec:fundamental_group}

Given a complex algebraic variety $X$, its complex points $X(\mathbb{C)}$ have the homotopy type of a finite CW-complex, which have finitely presented (and thus countable) fundamental groups. Consequently, $\pi_1(X^{an})$ is always a countable group.

Given a discrete countable group $G$, Tannaka duality for algebraic groups as in \cite{Deligne_Milne} allows one to recover only the \emph{pro-algebraic completion} of $G$, whereas the Tannaka duality theorem for geometric algebraic stacks as in \cite{Lurie_Tannaka_GS}) does not apply as $BG$ is not an algebraic stack in this case. Thus, even though the category of local systems on a complex algebraic variety is Tannakian, one can only hope to recover the pro-algebraic completion of the fundamental group from local systems on the space using usual algebraic Tannaka duality.

Denoting the category of liquid representations of $\pi_1(X^{an}, x)$ by $\text{LocSys}^{\text{Liq}}(X)$, and using instead the analytic Tannaka duality formalism developed in this paper, we have:

\begin{corollary}
    Let $X$ be a complex algebraic variety, with $x \in X(\mathbb{C)}$ a complex point. Let $U^{\text{Liq}}: \text{LocSys}^{\text{Liq}}(X) \to \text{Liq}_p(\mathbb{C})$ be the forgetful functor. Then $\pi_1(X^{an}, x) = \text{Aut}^\otimes(U^{\text{Liq}})$.
\end{corollary}

\begin{proof}
    $\pi_1(X^{an},x)$ being a discrete countable group means it is a globally finitely presented Stein complex Lie group by \Cref{countable_groups_example}. The result then follows from \Cref{thm:Stein_analytic_groups}.
\end{proof}

\begin{example}\label{Pro-algebraic_completion}
    To contrast this with the result that is output by usual algebraic Tannaka duality, we remark that if we take our complex algebraic variety to be $\mathbb{C}^*$, so that $\pi_1(\mathbb{C}^*, x) = \mathbb{Z}$, the Tannakian group of $\text{Rep}(\mathbb{Z})$ is the pro-algebraic completion of $\mathbb{Z}$, which is $\mathbb{C} \times T \times \hat{\mathbb{Z}}$, where $T$ is the pro-torus with character group $Hom(\mathbb{Z, \mathbb{C}^*)}$ \cite{BLMM_proalgebraic}. The result above instead produces from the category of liquid local systems $\mathbb{Z}$ itself.
\end{example}

\subsection{Reconstructing the Stokes groupoids}\label{subsec:stokes_groupoids}

In \cite{Gualtieri_Li_Pym_Stokes}, given a complex curve $X$ and an effective divisor $D$, the authors construct certain analytic Lie groupoids $\Pi_1(X,D)$ they call twisted fundamental groupoids. The category of flat connections on vector bundles on $X$ with singularities along $D$ is equivalent to the category of representations of the Lie \textit{algebroid} $T_X(-D)$, whose sections are vector fields on $X$ which vanish along $D$. These Lie algebroids are integrable to Lie groupoids, and $\Pi_1(X,D)$ denotes the source simply-connected integration of $T_X(-D)$. One can view this as a generalized Riemann-Hilbert correspondence for flat connections with singularities along a divisor (logarithmic singularities if the divisor is reduced, and irregular singularities if it is not). From this point of view, the generalized Riemann-Hilbert correspondence is the composition of the purely formal identification of meromorphic flat connections with representations of a Lie algebroid with the transcendental procedure of integrating Lie algebroids to Lie groupoids.
\[ \text{FlatConn}_X(D) \cong \text{Rep}(T_X(-D))\cong\text{Rep}(\Pi_1(X,D))\]

In particular, the local theory is given by flat connections on the (complex) affine line $\mathbb{A}^1$ with singularities along the divisor $D = n.0$, corresponding to an $n^{th}$ order pole at the origin, which by the above equivalence are given by representations of $\Pi_1(\mathbb{A}^1, n.0)$. The authors refer to this as the $n^{th}$ Stokes groupoid. The first Stokes groupoid $\Pi_1(\mathbb{A}^1,0)$ (which classifies connections with a logarithmic singularity at the origin) is identified with the action groupoid $\mathbb{C} \ltimes \mathbb{A}^1$ of $\mathbb{C}$ acting on $\mathbb{A}^1$ by exponentiation followed by multiplication, i.e. $\lambda . z = e^{\lambda}z$. 

One may consider the stack presented by this groupoid as the quotient stack $\mathbb{A}^1/\mathbb{G}_a$, which is \textit{not} an algebraic stack because of the occurrence of the exponential in the action, unlike the algebraic stack $\mathbb{A}^1/\mathbb{G}_m$ which is presented by the algebraic action groupoid $\mathbb{C}^* \ltimes \mathbb{A}^1$. Thus, the formalism of analytic stacks rather than algebraic stacks is required to talk about logarithmic connections in this manner.

We can consider the category $\text{QCoh}(\mathbb{A}^1/\mathbb{G}_a) = \text{Rep}^{\text{Liq}}(\mathbb{C}\ltimes \mathbb{A}^1)$ as a quasicoherent generalization of the category of flat connections on $\mathbb{A}^1$ with a logarithmic singularity at the origin, by the equivalence above.

\begin{theorem}\label{thm:Sto_1_example}
    Let $p^*: \text{QCoh}(\mathbb{A}^1/\mathbb{G}_a) \to \text{QCoh}(\mathbb{A}^1)$ be the pullback functor induced by $\mathbb{A}^1 \to \mathbb{A}^1/\mathbb{G}_a$. Then $\text{Aut}^{\otimes}(p^*) = \Pi_1(\mathbb{A}^1, 0)$.
\end{theorem}

\begin{proof}
    Since the diagonal of the stack $\mathbb{A}^1/\mathbb{G}_a$ is given by the action groupoid $\mathbb{C} \ltimes \mathbb{A}^1$ whose underlying space is $\mathbb{A}^2$, the diagonal is a globally finitely presented Stein space, and we are in the setting of \Cref{thm:TD_for_1_geometric_stacks}. Then taking automorphisms of $p^*$ we have:
    \[ \text{Aut}^\otimes (p^*: \text{QCoh}(\mathbb{A}^1/\mathbb{G}_a) \to \text{QCoh}(\mathbb{A}^1)) = \text{Aut}(\mathbb{A}^1 \to \mathbb{A}^1/\mathbb{G}_a) = \mathbb{C} \ltimes \mathbb{A} = \Pi_1(\mathbb{A}^1, 0)
    \]
    where the second last equality follows from the fact that such automorphisms are given by the diagonal.
\end{proof}

In fact, exactly the same argument generalizes to the $n^{th}$ Stokes groupoid, even though these are no longer given as action groupoids. Denote by $B\mathcal{G}$ the stack associated with an analytic groupoid.

\begin{theorem}\label{Sto_n_example}
Let $p^*: \text{QCoh}(B\Pi_1(\mathbb{A}^1, n.0))$ be the pullback functor induced by $p: \mathbb{A}^1 \to B\Pi_1(\mathbb{A}^1, n.0)$. Then $\text{Aut}^{\otimes}(p^*) = \Pi_1(\mathbb{A}^1, n.0)$.
\end{theorem}

\begin{proof}
    By Theorem 3.21 in \cite{Gualtieri_Li_Pym_Stokes}, the underlying space of arrows $\Pi_1(\mathbb{A}^1, n.0)$ can be identified with $\mathbb{A}^2$. Thus the diagonal is once again a globally finitely presented Stein space, and \Cref{thm:TD_for_1_geometric_stacks} applies. The rest of the argument is identical to the proof of \Cref{thm:Sto_1_example}.
\end{proof}

\subsection{Necessity of working with quasicoherent sheaves}

It is not possible to obtain identical results by working with just liquid \textit{coherent} sheaves. To see this, taking the stack $B\mathbb{Z}$, $\text{Coh}(B\mathbb{Z})$ corresponds to finite-dimensional liquid representations. Since $\mathbb{Z}$ is discrete, such liquid representations in fact correspond to abstract representations of $\mathbb{Z}$ (i.e. without any liquid or holomorphic structure). We are then in the setting of \Cref{Pro-algebraic_completion}, which shows that taking automorphisms of the forgetful functor only recovers the pro-algebraic completion of $\mathbb{Z}$.

\end{document}